\documentclass[11pt]{amsart}

\usepackage{amsmath}
\usepackage{amssymb}
\usepackage{graphicx}
\usepackage{color}

\newtheorem{theorem}{Theorem}[section]

\newtheorem{lemma}[theorem]{Lemma}
\newtheorem{proposition}[theorem]{Proposition}
\theoremstyle{definition}

\newtheorem{remark}[theorem]{Remark}

\numberwithin{equation}{section}
\newcommand{\diver}{\mathrm{div}\,}

\newcommand{\dx}{\,\mathrm{d}x}

\newcommand{\dl}{\displaystyle}

\newcommand{\na}{\nabla}

\newcommand{\dnu}{\,\mathrm{d}\nu}
\newcommand{\ra}{\,\rightarrow}

\newcommand{\loc}{{\mathrm{loc}}}

\newcommand{\N}{\mathbb{N}}
\newcommand{\R}{\mathbb{R}}

\def\ga{\alpha}     \def\gb{\beta}       \def\gg{\gamma}
       \def\gd{\delta}      
                         \def\vge{\varepsilon}
       \def\vgf{\varphi}    
            
\def\gm{\mu}        \def\gn{\nu}

     \def\Gd{\Delta}

\def\Gw{\Omega}              

\title[Weighted $L^p$-Hardy inequality]{On weighted $L^p$-Hardy inequality\\ on domains in $\R^n$}
\author{Divya Goel}
\address{Department of Mathematics\\
	Technion - Israel Institute of Technology\\
	Haifa 32000\\
	Israel }
\email{divyagoel2511@gmail.com; divya.goel@campus.technion.ac.il}
\author{Yehuda Pinchover}
\address{Department of Mathematics\\
	Technion - Israel Institute of Technology\\
	Haifa 32000\\
	Israel }
\email{pincho@technion.ac.il}
\author{Georgios  Psaradakis}
\address{University of Mannheim, Math IV\\
	Mannheim 68131, Germany}
\email{psaradakis@mail.uni-mannheim.de; psaradakis@outlook.com}

\dedicatory{Dedicated to Professor Shmuel Agmon}
\begin{document}
\begin{abstract}
We consider weighted $L^p$-Hardy inequalities involving the distance  to the boundary of a domain in the $n$-dimensional Euclidean space with nonempty boundary.
Using criticality theory, we give an alternative proof of the following result of F.~G.~Avkhadiev (2006)

\medskip

\noindent{\bf Theorem.} {\em
	Let $\Omega \subsetneqq \mathbb{R}^n$, $n\geq 2$, be an arbitrary domain, $1<p<\infty$ and  $\alpha + p>n$. Let $\mathrm{d}_\Omega(x) =\mathrm{dist}(x,\partial \Omega )$ denote the distance
	of a point $x\in \Omega$ to  $\partial \Omega$. Then the following Hardy-type inequality holds
	$$
	\int_{\Omega }\frac{|\nabla \varphi |^p}{\mathrm{d}_\Omega^{\alpha}}\,\mathrm{d}x \geq \left( \frac{\alpha +p-n}{p}\right)^p \int_{\Omega }\frac{|\varphi|^p}{\mathrm{d}_\Omega^{p+\alpha}}\,\mathrm{d}x  \qquad \forall \varphi\in C^{\infty }_c(\Omega),$$
	and the lower bound constant $\left( \frac{\alpha +p-n}{p}\right)^p$ is sharp.
}
\end{abstract}
\maketitle
\section{Introduction}\label{sec0}
Let $\Omega$ be a domain in ${\mathbb{R}}^n$, $n\geq 2$ with nonempty boundary, and let $\mathrm{d}_\Omega (x)=\mathrm{dist}(x,\partial \Omega )$ denote the distance
of a point $x\in \Gw$ to the boundary of $\Omega$. Fix $p\in (1,\infty )$. We say that the {\em $L^p$-Hardy inequality} is satisfied in $\Omega $
if there exists $c>0$ such that
\begin{equation}
	\label{mainhardy}
	\int_{\Omega }|\nabla \vgf|^p\dx \geq c \int_{\Omega }\frac{|\vgf|^p}{\mathrm{d}_\Omega^p}\dx  \qquad \mbox{ for all $\vgf\in C^{\infty }_c(\Omega)$}.
\end{equation}
The {\em $L^p$-Hardy constant} of $\Omega$ is
the best constant $c$ for inequality (\ref{mainhardy}) which is denoted  here by  $H_p(\Omega )$. It is a classical result that goes back to Hardy himself (see for example \cite{BEL, permalkuf}) that if
$n=1$ and $\Omega \subsetneqq\R$ is a bounded or unbounded interval, then the $L^p$-Hardy inequality holds and $H_{p}(\Omega)$ coincides with the widely known constant
$$
c_p=\biggl(\frac{p-1}{p}\biggr)^p.
$$
Recall that   if $\Omega $ is bounded and has a sufficiently regular boundary in ${\mathbb{R}}^n$, then the $L^p$-Hardy inequality holds and $0< H_p(\Omega )\le c_p$  (for instance, see \cite{anc,mamipi}).  Moreover, if $\Omega$ is convex, or more generally, if it is weakly mean convex, i.e., if $\Delta \mathrm{d}_\Gw\leq 0$ in the distributional sense in $\Omega$ (see \cite{gromov,giga and giovanni, LLL, psaradakis}), then $H_p(\Omega )=c_p $ \cite{barfilter, dam, mamipi}. On the other hand, it is also well-known (see for example \cite{BEL,permalkuf}) that if $\Omega ={\mathbb{R}}^n\setminus\{0\}$ and $p\ne n$, then the $L^p$-Hardy inequality holds and  $H_{p}(\Omega)$  coincides with the other widely  known constant
$$
c_{p,n}=\biggl|\frac{p-n}{p}\biggr|^p,
$$
which indicates  that  the $L^p$-Hardy inequality does not hold for ${\mathbb{R}}^n\setminus\{0\}$ if $p=n$.

In the present paper we study  a {\em weighted} $L^p$-Hardy inequality involving the distance function to the boundary. We give a new proof for the following result.
\begin{theorem}\label{main_thm}
	Let $\Gw \subsetneqq \R^n$ be an arbitrary domain, where $n\geq 2$. Fix $1<p<\infty$ and  $\ga + p>n$.  Then
	\begin{equation}\label{dp15}
		\int_{\Omega }\frac{|\nabla \vgf|^p}{\mathrm{d}_\Gw^{\ga}}\dx \geq \left( \frac{\ga +p-n}{p}\right)^p \int_{\Omega }\frac{|\vgf|^p}{\mathrm{d}_\Gw^{p+\ga}}\dx  \qquad \forall \vgf\in C^{\infty }_c(\Omega),
	\end{equation}
	and the lower bound constant $$c_{\ga,p,n}:= \left( \frac{\ga +p-n}{p}\right)^p$$ is sharp.
	In particular, for $p>n$ we have
	$$H_p(\Omega )\geq    c_{p,n}=\left( \frac{p-n}{p}\right)^p  \quad \mbox{for any domain } \Gw \subsetneqq \R^n.$$
\end{theorem}
\begin{remark}
		Theorem~\ref{main_thm} was proved by F.~G.~Avkhadiev in \cite{avk} using a cubic approximation of $\Gw$.
		One should note that J.~L.~Lewis \cite{Lewis} proved that \eqref{mainhardy} holds true (for $\ga=0$) with
		{\em a fixed positive  constant independent on} $\Gw$, and in \cite{Wannebo}, A.~Wannebo generalized Lewis' result to the case  $\ga + p>n$.
\end{remark}

We need the following version of the Harnack convergence principle which  will be used several times throughout the paper.
\begin{proposition}[Harnack convergence principle]\label{propdp1}
	Consider an exhaustion  $\{\Gw_i\}_{i\!=\!1}^\infty$ of smooth bounded domains such that
	$$  \left\{ x \in \overline{\Gw}~:~ \mathrm{d}_\Gw(x)>\frac{1}{i}\right\}  \subseteq   \Gw_i \Subset \Gw_{i+1}, \mbox{ and } \cup_{i \in \mathbb{N}}\Gw_i = \Gw.$$
	For each $i\in \N$, let $u_i$ be a positive (weak) solutions of the equation
	\begin{equation*}\label{dp2}
		-\diver(\mathrm{d}_{\Gw_i}^{-\ga} |\na u_i|^{p-2}\na u_i ) -\mu_i\frac{|u_i|^{p-2}u_i}{\mathrm{d}_{\Gw_i}^{\ga+p}}=0 \qquad \text{ in } \Gw_i
	\end{equation*}
	such that $u_i(x_0)=1$, where $x_0 \in \Gw_1$, and $\gm_i \in \R$.
	
	If  $\gm_i\to\gm$, then there exists $0<\gb<1$ such that, up to a subsequence, $\{u_i\}$ converges in $C^{0,\gb}_\loc(\Gw)$ to a positive (weak) solution $u\in W^{1,p}_\loc(\Gw)$  of the equation
	\begin{equation*}
		-\diver(\mathrm{d}_{\Gw}^{-\ga} |\na u|^{p-2}\na u ) -\mu\frac{|u|^{p-2}u}{\mathrm{d}_\Gw^{\ga+p}}=0 \qquad \text{ in } \Gw .
	\end{equation*}
\end{proposition}

\begin{proof}
	Since $\mathrm{d}_{\Gw_i}\to \mathrm{d}_{\Gw}$, the theorem follows directly from \cite[Proposition 2.7]{GP}.
\end{proof}
The  paper is organized as follows: In Section~\ref{sec2} we give our proof of Theorem~\ref{main_thm} while in Appendix we outline two alternative proofs.

\section{Proof of Theorem \ref{main_thm}}\label{sec2}
Our proof of Theorem~\ref{main_thm} is based on a simple  construction of a (weak) positive supersolutions to the associated Euler-Lagrange Lagrange equations
\begin{equation}\label{EL_eq}
	-\diver(\mathrm{d}_{\Gw}^{-\ga}|\nabla u|^{p-2}\nabla u) - \gm \frac{|u|^{p-2}u}{\mathrm{d}_{\Gw}^{\ga+p}} =0 \qquad \mbox{in } \Gw,
\end{equation}
for any $\gm<c_{\ga,p,n}$. Theorem~\ref{main_thm} then follows from the Harnack convergence principle (Proposition~\ref{propdp1}) together with the Agmon-Allegretto-Piepenbrink-type (AAP) theorem~\cite[Theorem~4.3]{pinpsa} which asserts that the Hardy inequality \eqref{dp15} holds true if and only if \eqref{EL_eq} admits a positive (super)solution for  $\gm=c_{\ga,p,n}$.
It seems that the method of the proof can be used to prove lower bounds for the best Hardy constant in different situations. 

\begin{proof}[Proof of Theorem \ref{main_thm}]	
	A direct computation shows that for any $y \in \R^n$,the function
	$$u_{y}(x):=|x-y|^{(\ga+p-n)/(p-1)}$$
	is a positive solution of the equation	$$-\diver(\mathrm{d}_{\Gw_y}^{-\ga}(x)|\nabla u|^{p-2}\nabla u)=0 \qquad \mbox{in } \Gw_y :=  \R^n\setminus \{ y \},$$
	where $\mathrm{d}_{\Gw_y}(x)=|x-y|$.
	
	Hence, using the supersolution construction \cite{depi}, it follows that
	$$v_y(x):= u_{y}^{(p-1)/p}(x) =|x-y|^{(\ga+p-n)/p}$$ is a positive solution of the equation
	$$-\diver(\mathrm{d}_{\Gw_y}^{-\ga}|\nabla u|^{p-2}\nabla u) - c_{\ga,p,n}\frac{|u|^{p-2}u}{\mathrm{d}_{\Gw_y}^{\ga+p}} =0 \qquad \mbox{in } \Gw_y.$$
	Moreover, it is known \cite{BEL} (see also \cite{avk06}) that  $c_{\ga,p,n}$ is the best constant for the inequality
	\begin{equation*}
		\int_{\Gw_y} \frac{|\nabla \vgf|^p}{|x-y|^{\ga}} \dx \geq \gm \int_{\Gw_y}\frac{|\vgf|^p}{|x-y|^{p+\ga}}\dx  \qquad \forall \vgf\in C^{\infty}_c(\Gw_y).
	\end{equation*}
	Hence, the lower bound for the Hardy constant for the functional inequality
	\begin{equation*}
		\int_{\Omega }\frac{|\nabla \vgf|^p}{\mathrm{d}_\Gw^{\ga}}\dx \geq \gm  \int_{\Omega }\frac{|\vgf|^p}{\mathrm{d}_\Gw^{p+\ga}}\dx  \qquad \forall \vgf\in C^{\infty }_c(\Omega),
	\end{equation*} in a domain $\Gw\subsetneqq \R^n$ is less or equal to $c_{\ga,p,n}$.

	It remains to prove that \eqref{EL_eq} admits positive supersolutions in $\Gw$ for $$\gm= \gm_\gd:= c_{\ga,p,n}-\gd>0, \qquad \forall\, 0<\gd<c_{\ga,p,n},$$ where $\Gw \subsetneqq \R^n$ is an arbitrary domain. We divide the proof into two steps.
	
	\textbf{Step 1:} Assume first that $\Gw$ is a smooth  bounded domain.
	Fix $\gd$ as above, and  choose  $\vge>0$ small enough such that
	\begin{align}\label{dp12}
		\vge<  \min \left\{ \left(\frac{ c_{\ga,p,n}}{\mu_\gd}\right)^{1/p} \!-1,\;  \frac{\mu_\gd(\ga+p-n)}{p|\ga| c_{\ga,p,n} }\right\}.
	\end{align}
	For  $x \in \Gw$, let $  P(x)\in \partial \Gw$ be the projection $x$ of into $\partial \Gw$ which is well defind for  a.e. $x\in \Gw$, that is, $|x-P(x)|=\mathrm{d}_\Gw(x)$. For any $y\in \partial \Gw$, consider the set
\begin{align*}
	D_{y,\vge}\!: =\!\Big\{ x\in \Gw \mid & ~ |x-y|\!< \!(1+ \vge) \mathrm{d}_\Gw(x),~ \cos(x-y,x-P(x))\!>\!1-\vge, \\
	&  \quad \text{ and }\mathrm{d}_\Gw(x)\!>\!\vge/2  \big\}.
\end{align*}
	If \begin{align}\label{dp14}
		\Gw_\vge= \{ x \in \Gw \mid \mathrm{d}_\Gw(x)>\vge\},
	\end{align}
	then $\dl\cup_{y \in \partial \Gw} D_{y,\vge}$ is an open covering of the compact set  $\overline{\Gw}_\vge$. Therefore, there exist $y_i,~i=1,2,\cdots,m $ such that $\Gw_\vge \subseteq \dl\cup_{i=1}^{m} D_{y_i,\vge}$
	We note that  $u_y$ is a positive supersolution to the equation
	\begin{equation*}
		-\diver(\mathrm{d}_{\Gw}^{-\ga} |\na u|^{p-2}\na u ) + \vge |\ga| k_{\ga,p,n}\frac{|u|^{p-2}u }{\mathrm{d}_{\Gw}^{\ga+p}}
		=0 \qquad \text{ in } D_{y,\vge}.
	\end{equation*}
	where $k_{\ga,p,n}:= \left( \frac{\ga+p-n}{p-1}\right)^{p-1}$. Indeed,  for $\ga\geq0$, 
	\begin{align*}
		-\diver(&\mathrm{d}_{\Gw}^{-\ga} |\na u_y|^{p-2}\na u_y ) 	 \\& = \ga \mathrm{d}_\Gw^{-\ga} k_{\ga,p,n}\left(\frac{ \na \mathrm{d}_\Gw\cdot (x-y)|x-y|^{\ga-n}}{\mathrm{d}_\Gw} -|x-y|^{\ga-n} \right)\\
		& \geq \ga \mathrm{d}_\Gw^{-\ga} k_{\ga,p,n}|x-y|^{\ga-n} \left(\frac{ |\na \mathrm{d}_\Gw||x-y|(1-\vge)}{\mathrm{d}_\Gw} -1\right)\\
		&\geq   \ga \mathrm{d}_\Gw^{-\ga} k_{\ga,p,n}|x-y|^{\ga-n}  \left( 1-\vge-1  \right)\\
		& = -\vge \ga \mathrm{d}_\Gw^{-\ga} k_{\ga,p,n}|x-y|^{\ga-n} \qquad  \mbox{in } D_{y,\vge}.
	\end{align*}
	Hence,
	\begin{align*}
		-\diver(\mathrm{d}_{\Gw}^{-\ga} |\na u_y|^{p-2}\na u_y )&  + \vge |\ga| k_{\ga,p,n}\frac{|u_y|^{p-2}u_y }{\mathrm{d}_{\Gw}^{\ga+p}}\\
		& \geq  \mathrm{d}_\Gw^{-\ga} k_{\ga,p,n}|x-y|^{\ga-n}  \left(-\vge\ga + \vge |\ga| \frac{|x-y|^p}{\mathrm{d}_{\Gw}^{p}}	
		\right)\\
		& \geq  \mathrm{d}_\Gw^{-\ga} k_{\ga,p,n}|x-y|^{\ga-n}  \left(-\vge\ga + \vge |\ga| 	
		\right)=0 \quad  \mbox{in } D_{y,\vge}.
	\end{align*}
	Similarly, for $\ga<0$
	\begin{align*}
		-\diver(&\mathrm{d}_{\Gw}^{-\ga} |\na u_y|^{p-2}\na u_y ) + \vge |\ga| k_{\ga,p,n}\frac{|u_y|^{p-2}u_y }{\mathrm{d}_{\Gw}^{\ga+p}}\\
		& \geq  \mathrm{d}_\Gw^{-\ga} k_{\ga,p,n}|x-y|^{\ga-n}  \left(   \frac{ \ga \na \mathrm{d}_\Gw\cdot (x-y)}{\mathrm{d}_\Gw} -\ga + \vge |\ga| \frac{|x-y|^p}{\mathrm{d}_{\Gw}^{p}}	
		\right)\\
		& \geq  \mathrm{d}_\Gw^{-\ga} k_{\ga,p,n}|x-y|^{\ga-n}  \left(  \frac{\ga |\na \mathrm{d}_\Gw||x-y|}{\mathrm{d}_\Gw} -\ga  + \vge |\ga| 	
		\right)\\
		& \geq  \mathrm{d}_\Gw^{-\ga} k_{\ga,p,n}|x-y|^{\ga-n}  \left( \ga(1+\vge) -\ga  + \vge |\ga| 	
		\right)\\
		& \geq  \mathrm{d}_\Gw^{-\ga} k_{\ga,p,n}|x-y|^{\ga-n}  \left( \vge\ga  + \vge |\ga| 	
		\right)
		=0  \qquad  \mbox{in } D_{y,\vge}.
	\end{align*}
	Now, the weak comparison principle  \cite[Lemma~5.1]{pinpsa} implies that $$u_\gd:= \min \{u_{y_i}\mid 1\leq i\leq m\}$$ is a supersolution  to the  equation
	\begin{equation}\label{dp1}
		-\diver(\mathrm{d}_{\Gw}^{-\ga} |\na u|^{p-2}\na u ) + \vge |\ga| k_{\ga,p,n}\frac{|u|^{p-2}u }{\mathrm{d}_{\Gw}^{\ga+p}} =0\quad  \text{ in } \Gw_\vge.
	\end{equation}
	\textbf{Claim 1:}  There exists a positive solution to the following equation
	\begin{equation}\label{dp11}
		-\diver(\mathrm{d}_{\Gw}^{-\ga} |\na u|^{p-2}\na u ) -
		\left(  \mu_\gd- \frac{\vge p|\ga|  c_{\ga,p,n}}{\ga+p-n}\right) \frac{|u|^{p-2}u }{\mathrm{d}_{\Gw}^{\ga+p}} =0  \text{ in } \Gw_\vge.
	\end{equation}
	Employing  the AAP-type theorem \cite[Theorem~4.3]{pinpsa},  it is enough to prove that there exists a positive  supersolution  to \eqref{dp11} in $\Gw_\vge$.   We use the supersolution construction \cite{depi} and prove that $v_\gd:= u_\gd^{(p-1)/p}$ is a supersolution to \eqref{dp11}.
	Using the fact that $u_\gd$ is a supersolution to \eqref{dp1}, we deduce that
	\begin{align*}
		& -\diver(\mathrm{d}_{\Gw}^{-\ga} |\na v_\gd|^{p-2}\na v_\gd ) -
		\left(  \mu_\gd- \frac{\vge p|\ga|  c_{\ga,p,n}}{\ga+p-n}\right)\frac{|v_\gd|^{p-2}v_\gd}{\mathrm{d}_\Gw^{\ga+p}}\\
		& =\! - \!\left(\!\frac{p-1}{p}\!\right)^{p-1}\!\!\!\!\!\diver\!(\mathrm{d}_{\Gw}^{-\ga} |\na u_\gd|^{p-2}\na u_\gd  u_\gd^{-(p-1)/p}) \!-\! \!
		\left(\!  \mu_\gd \!-\!  \frac{\vge p|\ga|  c_{\ga,p,n}}{\ga+p-n}\!\right)\!\! \frac{|u_\gd|^{\! (p-1)^2/p}}{\mathrm{d}_\Gw^{\ga+p}}\\
		& \geq \left(\frac{p-1}{p}\right)^{p}\mathrm{d}_{\Gw}^{-\ga} |\na u_\gd|^{p}  u_\gd^{-(2p-1)/p} -\mu_\gd\frac{|u_\gd|^{(p-1)^2/p}}{\mathrm{d}_\Gw^{\ga+p}}\\
		& = \frac{|u_\gd|^{(p-1)^2/p}}{\mathrm{d}_\Gw^{\ga+p}}\left[\left(\frac{p-1}{p}\right)^p\frac{|\na u_\gd|^pd^p}{u_\gd^p}- \mu_\gd \right] \qquad  \mbox{in } \Gw_{\vge}.
	\end{align*}
	Therefore, we need to prove that $\left[\left(\frac{p-1}{p}\right)^p\frac{|\na u_\gd|^pd^p}{u_\gd^p}- \mu_\gd \right]\geq 0$. Indeed,
	for a.e. $x \in \Gw_\vge$,  $u_\gd= u_{y_{i_0}}$ for some $i_0$  in a neighborhood of $x$.  Using the definition of $\vge$ and $D_{y,\vge}$, we get
	\begin{align*}
		&	\left(\frac{p-1}{p}\right)^p\frac{|\na u_\gd|^pd^p}{u_\gd^p}- \mu_\gd\\
		&= \left(\frac{p-1}{p}\right)^p\left(\frac{\ga+p-n}{p-1}\right)^p \frac{d^p}{|x-y_{i_0}|^p}- \mu_\gd\geq \frac{c_{\ga,p,n}}{(1+\vge)^p} - \mu_\gd
		>0.
	\end{align*}
	Hence, Claim 1 is proved.
	
	\medskip
	
	\textbf{Claim 2:} There exists a positive solution to the following equation
	\begin{equation}\label{dp13}
		-\diver(\mathrm{d}_{\Gw}^{-\ga} |\na u|^{p-2}\na u ) -\mu_\gd\frac{|u|^{p-2}u}{\mathrm{d}_\Gw^{\ga+p}}=0 \quad \text{ in } \Gw.
	\end{equation}
	Let $\vge_0>0$ be small enough such that \eqref{dp12} holds, and set $\vge_i := \min\{ \vge_0,\frac{1}{i}\}$. Clearly,  $ \Gw_{\vge_i} \Subset \Gw_{\vge_{i+1}} $ for $i$ large enough, and $\Gw= \cup_{i=1}^{\infty} \Gw_{\vge_i}$, where $\Gw_{\vge_i}$ is defined in \eqref{dp14}. Employing Claim~1, it follows that  for $i\geq 1$ there exists a positive solution $u_i$ to \eqref{dp11} in $\Gw_{\vge_i}$ satisfying $u_i(x_0)=1$.  In light of the Harnack convergence principle (Proposition \ref{propdp1}), it follows that Claim~2 holds.
	
	\medskip
	
	\textbf{Step 2: } Assume now that $\Gw$ is an arbitrary domain.
	Choose a  smooth compact exhaustion $\{ \Gw_i\}$ of $\Gw$. That  is, $\{ \Gw_i\}$ is a sequence of smooth bounded domains  such that  $ \Gw_i \Subset \Gw_{i+1}\Subset \Gw$, $\Gw= \cup_{i=1}^{\infty} \Gw_i$, and
	$$\max_{\substack{x\in \partial\Gw_i\cap B_i\\ y\in \partial\Gw\cap B_i}}\{\text{dist}(x,\partial \Gw) , \text{dist}(y,\partial \Gw_i) \} < \frac{1}{i},$$
	where $B_i=\{|x|<i\}$. Observe that $\mathrm{d}_{\Gw_i} \ra \mathrm{d}_{\Gw}$ a.e. in $\Gw$. Indeed,
	for $x\in \overline{\Gw_i\cap B_i}$ one has
	$$|\mathrm{d}_{\Gw}(x) - \mathrm{d}_{\Gw_i}\!(x)|=|\text{dist}(x,\partial \Gw)-\text{dist}(x,\partial \Gw_i)|<\frac1i \,.$$
	Invoking Claim 2, it follows that  for each $i\geq 1$, there exists  $u_i>0$ satisfying $u_i(x_0)=1$ and the equation
	\begin{equation*}
		-\diver(\mathrm{d}_{\Gw_i}^{-\ga} |\na u_i|^{p-2}\na u_i ) -
		\mu_\gd \frac{|u_i|^{p-2}u_i }{\mathrm{d}_{\Gw_i}^{\ga+p}} =0 \quad  \text{ in } \Gw_i.
	\end{equation*}
	
	Using again the Harnack convergence principle (Proposition \ref{propdp1}), we obtain a positive solution $u_\gd$ to \eqref{dp13} satisfying $u_\gd(x_0)=1$. Letting $\gd\to 0$, we get by Harnack convergence principle a positive solution $u_0$ to the equation
	\begin{equation}\label{dp21}
		-\diver(\mathrm{d}_{\Gw}^{-\ga} |\na u|^{p-2}\na u ) -
		c_{\ga,p,n} \frac{|u|^{p-2}u }{\mathrm{d}_{\Gw}^{\ga+p}} =0 \quad  \text{ in } \Gw
	\end{equation}
	that satisfies $u_0(x_0)=1$.
	In light of the AAP-type theorem we obtain the Hardy inequality
	\begin{align*}
	\int_{\Omega }\frac{|\nabla \vgf|^p}{\mathrm{d}_\Omega^{\alpha}}\,\mathrm{d}x \geq \left( \frac{\alpha +p-n}{p}\right)^p \int_{\Omega }\frac{|\vgf|^p}{\mathrm{d}_\Omega^{p+\alpha}}\,\mathrm{d}x  \qquad \forall \vgf\in C^{\infty }_c(\Omega). \qquad  \qedhere
	\end{align*}
\end{proof}

\appendix
\section{Different Proofs}
Here we  give  two alternative proofs of Theorem~\ref{main_thm}, both of them  do not use an exhaustion argument.  On the other hand, both rely on the following folklore lemma which is of independent interest, see for example, propositions 1.1.3. and 2.2.2. in \cite{CS}  (cf. \cite[Theorem~1.6]{LLL}, where the case of $C^2$-domains is discussed). 
\begin{lemma}\label{semiconcave}
Let $\Gw\! \subsetneqq \!\R^n$ be a domain. 

(i) The inequality  \[-\Delta \mathrm{d}_{\Gw} \geq -\frac{n-1}{\mathrm{d}_{\Gw}} \, ,\] holds true in the sense of distributions in $\Omega$.

(ii) Moreover,
\begin{equation}\label{lapdq}
\int_\Gw \nabla\psi\cdot\nabla{\rm d}_\Gw {\rm d}x
\geq -(n-1)\int_\Gw\frac{\psi}{{\rm d}_\Gw}{\rm d}x\quad\forall~\psi\in C^\infty_c(\Gw), \psi\geq0.
\end{equation} 
\end{lemma}

\begin{proof}\label{semiconcave remark}
(i) Since the function $|x|^2-\mathrm{d}_{\Gw}^2(x)$ is convex, it follows  that its distributional Laplacian is a nonnegative Radon measure (see \cite[Theorem 2-\S6.3]{evans and gariepy} and \cite[Lemma 2.1]{psaradakis} for the details). Hence,  
$$ \langle(n-1)-\mathrm{d}_{\Gw}\Delta \mathrm{d}_{\Gw}, \vgf \rangle  = \int_\Gw \vgf \dnu  \qquad \forall~\vgf \in C^\infty_c(\Gw),$$ 
where $\gn$ is a nonnegative Radon measure, and $\langle \cdot,  \cdot\rangle : \mathcal{D'}(\Gw)\times \mathcal{D}(\Gw)$ is the canonical duality pairing between distributions and test functions. 
Consequently, the distributional Laplacian of  $- \mathrm{d}_{\Gw}$ is itself a signed Radon measure $\mu$.  Thus,
\begin{equation*}\label{lapd}
 - \langle\Delta \mathrm{d}_{\Gw}, \psi \rangle\!=\!-\!\int_\Gw \!\!\Gd \psi {\rm d}_\Gw{\rm d}x
  \!=\! \int_\Gw\!\! \psi{\rm d}\mu
   \! \geq \! -(n-1)\!\!\int_\Gw \frac{\psi}{{\rm d}_\Gw}{\rm d}x \;\; \forall~\psi\in C^\infty_c(\Gw), \psi\geq0 .
\end{equation*}

(ii) Since $\nabla  {\rm d}_\Gw \in L^\infty(\Gw,\R^n)$, it follows that  
$$-\langle (\mathrm{d}_{\Gw})_{x_i, x_i}, \psi \rangle =  \int_\Gw  (\mathrm{d}_{\Gw})_{x_i} \psi_{x_i}\dx.$$  
Therefore,  $\Delta \mathrm{d}_{\Gw}$, the distributional divergence   of $\nabla  \mathrm{d}_{\Gw}$, satisfies
$$ - \langle\Delta \mathrm{d}_{\Gw}, \psi \rangle=  \int_\Gw  \nabla {\rm d}_\Gw\cdot \nabla \psi \dx \qquad\forall~\psi\in C^\infty_c(\Gw).$$
Hence,  
$$ \int_\Gw  \nabla {\rm d}_\Gw\cdot \nabla \psi \dx = - \langle\Delta \mathrm{d}_{\Gw}, \psi \rangle 
\geq  -(n-1)\int_\Gw\frac{\psi}{{\rm d}_\Gw}{\rm d}x   \;\; \forall~\psi\in C^\infty_c(\Gw), \psi\geq0 . \qedhere$$  
\end{proof}
\begin{lemma}\label{a3}
	Let $\Gw\! \subsetneqq \!\R^n$ be a domain. Let
	$$1<p<\infty, \quad  \ga\in \R^n, \quad \mbox{and }\;\;    0<\gg < \frac{\ga+p-n}{p-1}\,.$$
	 Then
	$\mathrm{d}_{\Gw}^\gg$ is a (weak) positive supersolution of the equation
	\begin{equation*}
		-\diver(\mathrm{d}_{\Gw}^{-\ga} |\na u|^{p-2}\na u)- C_{\ga,p,n,\gg} \frac{|u|^{p-2}u}{\mathrm{d}_{\Gw}^{p+\ga}}=0   \quad \text{ in } \Gw,
	\end{equation*}
where $C_{\ga,p,n,\gg} := |\gg|^{p-1}(\ga-n+1-(\gg-1)(p-1))>0$.

%
\end{lemma}
\begin{proof}
Using \eqref{lapdq} we obtain
\begin{multline*}
\int_{\Omega }\!\!\mathrm{d}_{\Gw}^{-\ga}|\nabla( \mathrm{d}_{\Gw}^\gg)|^{p-2}\nabla( \mathrm{d}_{\Gw}^\gg) \!\cdot \!\nabla \vgf \mathrm{d}x \!
=\!  |\gg|^{p-2}\gg \!\! \int_{\Omega } \!\!\mathrm{d}_{\Gw}^{(\gg-1)(p-1)-\ga} \nabla \mathrm{d}_{\Gw} \! \cdot \! \nabla \vgf \mathrm{d}x\\
=\!|\gg|^{p-1}\!\!\!  \int_{\Omega }\!\!\! \left(\! \nabla( \mathrm{d}_{\Gw}) \! \cdot \! \nabla (\mathrm{d}_{\Gw}^{\!(\gg-1)(p-1)-\ga} \!\! \vgf) \!-\! ((\gg \! - \!1)(p \! - \! 1)-\ga)\mathrm{d}_{\Gw}^{(\!\gg-1)(p-1)-\ga-1}\! \vgf \!\right) \!\mathrm{d}x\\
\geq C_{\ga,p,n,\gg} \int_{\Omega } \mathrm{d}_{\Gw}^{{(\gg-1)(p-1)-\ga-1}}\vgf \,\mathrm{d}x  . \qedhere
\end{multline*}
	\end{proof}
\begin{remark}
	Observe that $$ C_{\ga,p,n}= \max\left\{ C_{\ga,p,n,\gg} \mid {\gg \in \left(0,\frac{\ga+p-n}{p-1}\right)}\right\} ,$$
and the maximum is obtained with $\gg=(\ga+p-n)/p$.
\end{remark}

\begin{proof}[Alternative proof of Theorem~\ref{main_thm} I]
Using  Lemma \ref{a3} for  $\gg=(\ga+p-n)/p$, we deduce that $\mathrm{d}_{\Gw}^{(\ga+p-n)/p}$ is positive (weak) supersolution to \eqref{dp21}. Consequently, the AAP-type theorem \cite[Theorem~4.3]{pinpsa} implies the Hardy-type inequality \eqref{dp15}.
\end{proof}

 \begin{proof}[Alternative proof of Theorem~\ref{main_thm} II] 
Let $\Gw\! \subsetneqq \!\R^n$ be a domain, and fix  $s>n$. Using Lemma~\ref{semiconcave},  the following $L^1$-Hardy inequality is proved in  \cite[Theorem 2.3]{psaradakis}: 
\begin{align}\label{dp17}
\int_\Gw \frac{|\na \vgf|}{\mathrm{d}_{\Gw}^{s-1}}\dx\geq (s-n)\int_\Gw \frac{| \vgf |}{\mathrm{d}_{\Gw}^{s}}\dx \qquad \forall~ \vgf \in C_c^\infty(\Gw).
\end{align}
  Substituting $\vgf=|\psi|^p$  in \eqref{dp17} and using H\"older inequality, we obtain
\begin{align*}
	\frac{s-n}{p}\int_\Gw \frac{| \psi|^p}{\mathrm{d}_{\Gw}^{s}} \dx &\leq \int_\Gw \frac{|\psi|^{p-1}|\na \psi|}{\mathrm{d}_{\Gw}^{s-1}} \dx = \int_\Gw \frac{| \psi|^{p-1}}{\mathrm{d}_{\Gw}^{s-s/p}}\frac{| \na \psi|}{\mathrm{d}_{\Gw}^{s/p-1}} \dx\\[2mm]
	&\leq \left(\int_\Gw \frac{| \psi|^p}{\mathrm{d}_{\Gw}^{s}}  \right)^{1-1/p}\left(\int_\Gw \frac{| \na \psi|^p}{\mathrm{d}_{\Gw}^{s-p}}  \right)^{1/p} \qquad \forall~ \psi \in C_c^\infty(\Gw).
\end{align*}
Hence,  for $s= \ga+p$, we get \eqref{dp15}.
\end{proof}

\begin{center}
	{\bf Acknowledgements}
\end{center}
D.~G. and Y.~P.  acknowledge  the  support  of  the  Israel  Science Foundation (grant  637/19) founded by the Israel Academy of Sciences and Humanities.

\end{document}